\documentclass[a4paper,12pt,reqno]{amsart}
\usepackage{amsmath, amssymb,amsthm}

\nonstopmode \numberwithin{equation}{section}
\newtheorem{theorem}{Theorem}[section]

\newtheorem{corollary}{Corollary}[section]

\newtheorem{remark}{Remark}[section]

\begin{document}
\bibliographystyle{amsplain}

\title{{{
Order of Starlikeness and Convexity of certain integral transforms using duality techniques
}}}

\author{
Satwanti Devi
}
\address{
Department of  Mathematics  \\
Indian Institute of Technology, Roorkee-247 667,
Uttarkhand,  India
}
\email{ssatwanti@gmail.com}

\author{
A. Swaminathan
}
\address{
Department of  Mathematics  \\
Indian Institute of Technology, Roorkee-247 667,
Uttarkhand,  India
}
\email{swamifma@iitr.ernet.in, mathswami@gmail.com}

\bigskip

\begin{abstract}
For $\alpha\geq 0$, $\beta<1$ and $\gamma\geq 0$, the class
$\mathcal{W}_{\beta}(\alpha,\gamma)$ satisfies the condition
\begin{align*}
{\rm Re\,} \left( e^{i\phi}\left((1-\alpha+2\gamma)f/z+(\alpha-2\gamma)f'+
\gamma zf''-\beta\right)\frac{}{}\right)>0, \quad \phi\in {\mathbb{R}},{\,}z\in {\mathbb{D}};
\end{align*}
is taken into consideration. The Pascu class  of $\xi$-convex
functions of order $\sigma$ $(M(\sigma,{\,}\xi))$, having
analytic characterization
\begin{align*}
{\rm Re\,}\frac{\xi z(zf'(z))'+(1-\xi)zf'(z)}{\xi zf'(z)+(1-\xi)f(z)}>\sigma,\quad 0\leq \sigma< 1,\quad z\in {\mathbb{D}},
\end{align*}
unifies starlike and convex functions class of order $\sigma$.
The admissible and sufficient conditions on $\lambda(t)$ are
investigated so that the integral transforms
\begin{align*}
V_{\lambda}(f)(z)= \int_0^1 \lambda(t) \frac{f(tz)}{t} dt,
\end{align*}
maps the function from $\mathcal{W}_{\beta}(\alpha,\gamma)$
into $M(\sigma,{\,}\xi)$. Further several interesting
applications, for specific choice of $\lambda(t)$ are discussed
which are related to the classical integral transform.
\end{abstract}



\maketitle

\pagestyle{myheadings}
\markboth{
}{
Order of starlikeness and convexity of integral transforms
}

\section{introduction}
Consider the class $\mathcal{A}$ of all normalized and analytic
function $f$ satisfying $f(0)=f'(0)-1=0$, in the open unit disk
$\mathbb{D}=\{ z\in\mathbb{C}: |z|<1\}$. Let $\mathcal{S}$
denotes the subclass of $\mathcal{A}$ consisting of the
univalent functions in $\mathbb{D}$. A function $f(z) \in
{\mathcal {S}}$ is said to be starlike $(\mathcal{S}^{\ast}),$
if $f(\mathbb{D})$ is a domain which is starlike with respect
to the origin. Further generalization of the class
$\mathcal{S}^{\ast}$ is the class $\mathcal{S}^{\ast}(\sigma)$
having analytic characterization
\begin{align*}
\mathcal{S}^\ast(\sigma)=\left\{f\in\mathcal{S}:{\,}{\rm{Re}}{\,}
\left(\dfrac{zf^\prime(z)}{f(z)}\right)>\sigma;\quad 0\leq\sigma<1,{\,} z\in\mathbb{D}\right\}.
\end{align*}
Corresponding to the class $\mathcal{S}^\ast(\sigma)$, is the
class of convex functions of order $\sigma$,
$(\mathcal{C}(\sigma))$, satisfying the Alexander transform
\cite{DU} that when
$f\!\in\!\mathcal{C}(\sigma)\!\Longleftrightarrow\! zf'\!\in
\!\mathcal{S}^\ast(\sigma)$.
Note that ${\mathcal{S}}^\ast(0)= {\mathcal{S}}^\ast$ and ${\mathcal{C}}(0)={\mathcal{C}}$.
For a function $f\in\mathcal{A}$,
$M(\sigma,\xi)$ denotes the Pascu class of $\xi$-convex
functions of order $\sigma<1$, if it satisfies the analytic
condition \cite{Pascu}
\begin{align*}
{\rm Re}{\,} \dfrac{\xi z(zf^{\prime}(z))^{\prime}+(1-\xi)zf^{\prime}(z)}
{\xi zf^{\prime}(z)+(1-\xi)f(z)}>\sigma,\quad 0\leq\xi\leq1,
\end{align*}
or equivalently
\begin{align*}
\xi zf^{\prime}(z)+(1-\xi)f(z)\in \mathcal{S}^\ast(\sigma)
\end{align*}
Note that $M(\sigma,0)\equiv \mathcal{S}^{\ast}(\sigma)$ and
$M(\sigma,1)\equiv \mathcal{C}(\sigma)$ which implies that the
Pascu class $M(\sigma,\xi)$ unifies the class of starlike and
convex functions of order $\sigma$ i.e., it is the convex
combination of the class of starlike and convex  functions. It is interesting to note that the class $M(\sigma,\xi)$
also contains some non-univalent functions.

If $f(z)=\displaystyle\sum_{n=0}^{\infty}a_n z^n$ and
$g(z)=\displaystyle\sum_{n=0}^{\infty}b_n z^n$ are in
$\mathcal{A}$ then the convolution (Hadamard product) of $f$
and $g$ is given by
\begin{align*}
h(z):=(f\ast g)(z)=\sum_{n=0}^{\infty}a_nb_n z^n.
\end{align*}

For a non-negative and real-valued integrable function $\lambda(t)$, satisfying $\displaystyle\int_0^1\lambda(t) dt=1$
and a function $f\in\mathcal{A}$, R. Fournier and S. Ruscheweyh \cite{Four-rus-extremal} introduced the integral operator
\begin{align}\label{eq-lambda-operator}
F(z):= V_{\lambda}(f)(z)=\int_0^1 \lambda(t) \dfrac{f(tz)}{t}dt=z\int_0^1 \dfrac{\lambda(t)}{1-tz}dt\ast\dfrac{f(z)}{z}.
\end{align}
For specific choice of $\lambda(t)$, integral operator
\eqref{eq-lambda-operator} reduces to the well-known operators
such as Bernardi, Komatu, Hohlov operators and several other
operators that are to be discussed in Section $\ref{sec-application}$. See \cite{Swami M, Four-rus-extremal, kim&ronning,
rag M} and references therein for these literature of these operators.

For given $\alpha \geq0$, $\gamma \geq 0$ and $\beta<1$, Ali et
al. \cite{Abeer S*} defined the class
{\small{
\begin{align*}
\mathcal{W}_\beta(\alpha, \,\gamma)\!=\!
\left\{\dfrac{}{}f\!\in\!\mathcal{A}:{\rm Re}\left(e^{i\phi}\left((1-\alpha+2\gamma)\dfrac{f(z)}{z}+(\alpha-2\gamma)
f^\prime(z)+\gamma zf^{\prime\prime}(z)-\beta\right)\right)>0, z\in\mathbb{D}\right\}
\end{align*}
}} for some $\phi\in \mathbb{R}$. Many authors applied the
duality theory \cite{Rus, Ru75} on this class and its
particular cases. They obtained relation between $\beta$ and
$\lambda(t)$ so that the integral operator given in
\eqref{eq-lambda-operator} is univalent or belongs to
$M(\sigma,{\,}\xi)$ for particular values of $\sigma$ and
$\xi$. Initially this work was motivated by R. Fournier and
S. Ruscheweyh \cite{Four-rus-extremal} by obtaining the conditions
so that $V_\lambda(\mathcal{W}_\beta (1,{\,} 0))\!\in\! M(0,0)$.
Further, the conditions under which
$V_\lambda(\mathcal{W}_\beta (\alpha,{\,} \gamma))\!\in\! M(0,0)$
was discussed by the second author with R.M. Ali et al. in \cite{Abeer S*} and the results
corresponding to convexity case i.e.,
$V_\lambda(\mathcal{W}_\beta (\alpha,{\,} \gamma))\!\in\! M(0,1)$
was given by R.M. Ali et al. in \cite{Mahnaz C}. Generalization of the results given in \cite{Abeer S*} and \cite{Mahnaz C} exhibiting
the conditions under which $V_\lambda(\mathcal{W}_\beta (\alpha,{\,} \gamma))\!\in\! M(0,\xi)$
was obtained by the authors of the present work in \cite{Swami M}. Note that the present work is not the direct
extension of the results given in \cite{Swami M} as the conditions obtained in Section $\ref{sec-main-result}$ differs from the respective
one given in \cite{Swami M}. Further details in this regard are provided in Section $\ref{sec-application}$.

Recently S. Verma et al. \cite{sarika-S*} (see also
\cite{Suzeini-starlike}) investigated the constraints such that
$V_\lambda(\mathcal{W}_\beta (\alpha,{\,} \gamma))\!\in\!
M(\sigma, 0)$ and corresponding convexity results was given by
R. Omar et al. \cite{Suzeini-convex} (see also \cite{sarika-C})
under which $V_\lambda(\mathcal{W}_\beta (\alpha,{\,}
\gamma))\!\in\! M(\sigma,1)$. For further details in this
direction we refer to \cite{Abeer S*, Mahnaz C, Swami M, rag M}
and references therein.

The main aim of this work is to investigates the condition on
$\lambda(t)$ using duality technique so that the integral
transform, $V_\lambda(f)\!\in\! M(\sigma, \xi)$ whenever
$f\!\in\! \mathcal{W}_\beta(\alpha,\gamma)$.
This requires certain preliminaries that are outlined in Section $\ref{sec-pascu-order-prelim}$.
In Section \ref{sec-main-result}, the necessary and sufficient conditions
are derived which ensures that $V_\lambda(f)(z)$ carries the
function $f(z)$ from $\mathcal{W}_\beta(\alpha,\gamma)$ into
$M(\sigma,\xi)$. The criteria for
$V_\lambda(\mathcal{W}_\beta(\alpha,\gamma))$ to be in
$M(\sigma, \xi)$ is simplified for further applications. In Section
\ref{sec-application}, using the sufficient condition, several
interesting applications for specific choice of $\lambda(t)$
are discussed.
\section{preliminaries}\label{sec-pascu-order-prelim}
We consider some of the preliminaries that are useful for further
discussion. Considering two constants $\mu\geq0$ and
$\nu\geq0$ which was introduced in \cite{Abeer S*}, satisfing
\begin{align}\label{eq-mu+nu}
\mu+\nu=\alpha-\gamma \quad \text{and} \quad\mu\nu=\gamma.
\end{align}

For $\gamma=0$, choose $\mu=0$, so $\nu=\alpha \geq 0$. Now for
the case $\alpha=1+2\gamma$, equation \eqref{eq-mu+nu} gives
$\mu+\nu=1+\gamma=1+\mu\nu$, or $(\mu-1)(1-\nu)=0$ which give
rise to two cases:
\begin{itemize}
\item[{\rm{(i)}}] If $\gamma>0,$ then for $\mu=1$, gives
    $\nu=\gamma$.
\item[{\rm{(ii)}}] If $\gamma=0$, then for $\mu=0$, gives
    $\nu=\alpha=1$.
\end{itemize}
Note: Since the case $\gamma=0$ is considered in \cite{rag M}, we will only consider
the case when $\gamma>0$.

Let
\begin{equation}\label{phi-series}
\phi_{\mu,{\,}\nu}(z)=1+\sum_{n=1}^{\infty }\dfrac{(n\mu+1)(n\nu +1)}{(n+1)}z^{n},
\end{equation}
and
\begin{eqnarray}\label{psi-series}
\psi_{\mu,{\,}\nu}(z)=\phi_{\mu,{\,}\nu}^{-1}(z)=1+\sum\limits_{n=1}^{\infty }\dfrac{(n+1)}{(n\mu +1)(n\nu+1)}z^{n}
=\int_0^1\int_0^1\frac{1}{(1-s^\mu t^\nu  z)^2}{\,}ds{\,}dt,
\end{eqnarray}
where $\phi_{\mu,{\,}\nu}^{-1}$ gives the convolution inverse
of $\phi_{\mu,{\,}\nu}$
such that $\phi_{\mu,{\,}\nu}\ast{\,} \phi_{\mu,{\,}\nu}^{-1}=1/(1-z)$.\\
For $\gamma=0$, implies $\mu=0$ and $\nu=\alpha$. So it become clear that
\begin{align*}
\psi_{0,{\,}\alpha}(z)&=1+\sum\limits_{n=1}^{\infty}\frac{n+1}{n\alpha+1}z^n
=\int_0^1\frac{1}{(1-t^\alpha z)^2}{\,} dt.
\end{align*}
The case $\gamma >0$, implies $\mu >0$ and $\nu >0$. Now
changing the variables $u =t^{\nu }$ and $v =s^{\mu }$ will
give
\begin{align*}
\psi_{\mu,{\,}\nu}(z)=\frac{1}{\mu \nu
}\int_{0}^{1}\int_{0}^{1}\frac{u^{1/\nu -1}v^{1/\mu
-1}}{(1-uvz)^{2}}dudv.
\end{align*}
Therefore, we can write $\psi_{\mu,{\,}\nu}$ as
\begin{align*}
\psi_{\mu,{\,}\nu}(z)=\displaystyle\left\{
\begin{array}{cll}&\displaystyle\dfrac{1}{\mu\nu}\int_0^1\int_0^1\dfrac{u^{1/\nu-1}v^{1/\mu-1}}{(1-uvz)^2}dudv,
\quad \gamma>0,\\ \\
&\displaystyle\quad \quad \int_0^1\dfrac{1}{(1-t^\alpha z)^2}dt, \quad\quad \quad \quad{\,} \gamma=0,{\,}
\alpha\geq0.
\end{array}\right.
\end{align*}
Consider $g(t)$ to be the solution of initial-value problem
\begin{align}\label{de-g}
\frac{d}{dt}t^{1/\nu}(1+g(t))=\left\{
\begin{array}{cll}&\displaystyle
\dfrac{2}{\mu\nu}t^{1/{\nu-1}}\int_0^1 \dfrac{1-\sigma(1+st)}{(1-\sigma)(1+st)^2}s^{1/\mu-1}ds, \quad
\gamma>0,\\\\
&\displaystyle\quad \quad \dfrac{2}{\alpha}t^{1/\alpha-1}\dfrac{1-\sigma(1+t)}{(1-\sigma)(1+t)^2}, \quad\quad\quad \quad
\quad {\,}\gamma=0,{\,} \alpha>0,
\end{array}\right.
\end{align}
satisfying $g(0)=1$. Solution of differential equation
\eqref{de-g} is given by
\begin{align}\label{int-g(t)}
g(t)=\left\{
\begin{array}{cll}&\displaystyle
\dfrac{2}{\mu\nu}\int_0^1\int_0^1 \dfrac{1-\sigma(1+swt)}{(1-\sigma)(1+swt)^2}s^{1/\mu-1}w^{1/\nu-1}dsdw-1,\quad
\gamma>0,\\\\
&\displaystyle \quad\quad\dfrac{2}{\alpha}t^{-1/\alpha}\int_0^t\dfrac{1-\sigma(1+s)}{(1-\sigma)(1+s)^2}s^{1/\alpha-1}ds-1, \quad
\quad \quad \quad\gamma=0,{\,} \alpha>0.
\end{array}\right.
\end{align}
Further consider $q(t)$ be the solution of the initial-value
problem
\begin{align}\label{de-q}
\dfrac{d}{dt}(t^{1/\nu}q(t))=\left\{
\begin{array}{cll}&\displaystyle\dfrac{1}{\mu\nu}
t^{1/{\nu-1}}\int_0^1\dfrac{1-\sigma-(1+\sigma)st}{(1-\sigma)(1+st)^3}{s^{1/\mu-1}}ds, \quad
\gamma>0,\\\\
&\displaystyle \quad \quad\dfrac{1}{\alpha}{t^{1/\alpha-1}}\dfrac{1-\sigma-(1+\sigma)t}{(1-\sigma)(1+t)^3},  \quad\quad
\quad \quad \quad\gamma=0,{\,} \alpha>0,
 \end{array}\right.
\end{align}
satisfying $q(0)=0$. Solving differential equation \eqref{de-q}
gives
\begin{align}\label{int-q(t)}
q(t)=\left\{
\begin{array}{cll}&\displaystyle
\dfrac{1}{\mu\nu}\int_0^1\int_0^1 \dfrac{1-\sigma-(1+\sigma)swt}{(1-\sigma)(1+swt)^3}s^{1/\mu-1}w^{1/\nu-1}dsdw, \quad
{\,}\gamma>0,\\\\
&\displaystyle\quad\quad\dfrac{1}{\alpha}t^{-1/\alpha}\int_0^t\dfrac{1-\sigma-(1+\sigma)s}{(1-\sigma)(1+s)^3}s^{1/\alpha-1}ds,\quad\quad\quad\quad
\quad\gamma=0,{\,} \alpha>0.
\end{array}\right.
\end{align}
S. Verma et al. \cite{sarika-S*} and R. Omar et al.
\cite{Suzeini-convex} (see also \cite{Suzeini-starlike,
sarika-C}) have established the necessary and sufficient
conditions under which the integral operator $V_\lambda(f(z))$
carries the function $f(z)$ from $\mathcal{W}_\beta(\alpha,
\gamma)$, to the classes $\mathcal{S^*}(\sigma)$ and
$\mathcal{C}(\sigma)$, respectively, which are given in the
following two results.
\begin{theorem}\label{sarika-W-class-S*}\cite{Suzeini-starlike, sarika-S*}
If $\mu\geq 0$, $\nu\geq 0$ satisfies \eqref{eq-mu+nu}, and
$\beta < 1$ is given by
\begin{align*}
\frac{\beta}{1-\beta}=-\int_0^1\lambda(t)g(t)dt,
\end{align*}
where $g(t)$ is the solution of differential equation
\eqref{de-g}. Assume further that
$t^{1/\nu}\Lambda_{\nu}(t)\rightarrow 0$, and
$t^{1/\mu}\Pi_{\mu,{\,} \nu}(t)\rightarrow 0$ as $t\rightarrow
0^+.$ Then $F(z):=V_\lambda(\mathcal{W}_\beta(\alpha,
\gamma))\in\mathcal{S^*}(\sigma)$ if and only if
\begin{align}\label{sarika-W-class-S*-result}
\left\{
\begin{array}{cll}&\displaystyle
{\rm Re\,}{\int_0^1\Pi_{\mu,{\,}\nu}(t)t^{1/\mu-1}\left(
\dfrac{h_\sigma(tz)}{tz}-\dfrac{1-\sigma(1+t)}{(1-\sigma)(1+t)^2}\right)dt}\geq 0, & \gamma >0,\\
&\displaystyle{\rm Re\,}{\int_0^1\Pi_{0,{\,}\alpha}(t)t^{1/\alpha-1}\left(
\dfrac{h_\sigma(tz)}{tz}-\dfrac{1-\sigma(1+t)}{(1-\sigma)(1+t)^2}\right)dt}\geq 0, & \gamma=0,
 \end{array}\right.
\end{align}
where $\Lambda_\nu$, $\Pi_{\mu, \,\nu}$ and $h_\sigma$ are
defined as
\begin{align}\label{eqn-lambda-nu}
\Lambda_\nu(t)=\int_t^1\frac{\lambda(x)}{x^{1/\nu}}dx,\ \ \nu>0,
\end{align}
\begin{align}\label{eqn-Pi-nu-mu}
\Pi_{\mu, \,\nu}(t)=\left\{
\begin{array}{cll}&\displaystyle
\int_t^1\Lambda_\nu(x)x^{1/\nu-1-1/\mu}dx,
\quad \gamma>0 \quad(\mu>0,\nu>0),\\
&\displaystyle \quad\quad\quad\quad\Lambda_\alpha(t) ,\quad\quad\quad\quad \gamma=0 \quad(\mu=0,
\nu=\alpha>0),
 \end{array}\right.
\end{align}
and
\begin{align}\label{h_sigma(z)}
h_\sigma(z)=\frac{z\left(1+\dfrac{\epsilon+2\sigma-1}{2(1-\sigma)}z\right)}{(1-z)^2}, \quad |\epsilon|=1
\end{align}
respectively.
\end{theorem}
\begin{theorem}\label{sarika-W-class-C}\cite{Suzeini-convex, sarika-C}
If $\mu\geq 0$, $\nu\geq0$ satisfies $\eqref{eq-mu+nu}$, and
$\beta< 1$ is given by
\begin{align*}
\frac{\beta-1/2}{1-\beta}=-\int_0^1\lambda(t)q(t)dt,
\end{align*}
where $q(t)$ is the solution of differential equation
$(\ref{de-q})$. Further $\Lambda_\nu(t)$,
$\Pi_{\mu,{\,}\nu}(t)$ and $h_\sigma$ are given by equation
$\eqref{eqn-lambda-nu}$, $(\ref{eqn-Pi-nu-mu})$ and
$\eqref{h_sigma(z)}$. Assume that
$t^{1/\mu}\Lambda_\nu(t)\rightarrow 0$, and $t^{1/\nu}\Pi_{\mu,
\,\nu}(t)\rightarrow 0$ as $t\rightarrow 0^+.$ Then
\begin{align}\label{sarika-W-class-C-result}
\left\{
\begin{array}{cll}&\displaystyle
{\rm Re\,}
{\int_0^1\Pi_{\mu,{\,}\nu}(t)t^{1/\mu-1}\left(h'_\sigma(tz)-\dfrac{1-\sigma-(1+\sigma)t}{(1-\sigma)(1+t)^3}\right)dt}\geq
0, \quad \gamma >0,
\\
&\displaystyle{\rm Re\,}
{\int_0^1\Pi_{0, {\,}\alpha}(t)t^{1/\alpha-1}\left(h_\sigma'(tz)-\dfrac{1-\sigma-(1+\sigma)t}{(1-\sigma)(1+t)^3}\right)dt}\geq
0, \quad \gamma=0,
\end{array}\right.
\end{align}
if and only if $\ F(z):=V_\lambda(\mathcal{W}_\beta(\alpha,
\gamma))\in\mathcal{C}(\sigma)$.
\end{theorem}
\section{main results}\label{sec-main-result}
In the following result the condition under which the integral
transform $V_\lambda(f)(z)$ carries the function $f(z)$ from
the class $\mathcal{W}_{\beta}(\alpha,\gamma)$ to
$M(\sigma,\xi)$ is obtained.
\begin{theorem}\label{Thm-main-pascu}
Let $\mu\geq0$ , $\nu\geq0$, satisfies \eqref{eq-mu+nu} and
$\beta <1$ is given as
\begin{align}\label{beta-cond-pascu}
\dfrac{\beta}{(1-\beta)}=-\int_0^1\! \lambda(t)\left(\dfrac{}{}(1-\xi)g(t)+\xi(2q(t)-1)\right)dt,
\end{align}
where $g(t)$ and $q(t)$ are defined by the differential
equation \eqref{de-g} and \eqref{de-q} respectively. Further
$\Lambda_\nu (t)$, $\Pi_{\mu, \,\nu}$ and $h_\sigma$ are given
by \eqref{eqn-lambda-nu}, \eqref{eqn-Pi-nu-mu} and
\eqref{h_sigma(z)}. Assume that $t^{1/\nu}\Lambda_\nu
(t)\rightarrow 0$ and $t^{1/\mu}\Pi_{\mu, \,\nu}(t)\rightarrow
0$ as $t\rightarrow 0^+$. Then $\mathcal{M}_{\Pi_{\mu,
\,\nu}}(h_\sigma)\geq 0$ if and only if
$F\!:=\!V_{\lambda}(\mathcal{W}_{\beta}(\alpha,\gamma))\!\in\!
M(\sigma,\xi)$ or $(\xi zF^\prime + (1-\xi)F)\!\in\!
\mathcal{S}^\ast(\sigma)$, where
\begin{align*}
\mathcal{M}_{\Pi_{\mu,\,\nu}}(h_\sigma)=
\left\{
\begin{array}{cll}&\displaystyle
\int_{0} ^{1} t^{{1/{\mu}} -1} \Pi_{\mu, \,\nu}(t){\,}{\mathcal L}_{\sigma,{\,} \xi, \,z}(t){\,}dt, \quad\gamma>0,
\\
&\displaystyle\int_{0} ^{1} t^{{1/{\alpha}} -1} \Pi_{0, \,\alpha}(t){\,}{\mathcal L}_{\sigma,{\,} \xi, \,z}(t){\,}dt, \quad\gamma=0,
\end{array}\right.
\end{align*}
and
\begin{align*}
{\mathcal L}_{\sigma,{\,}\xi,\,z}(t)\!=\!(1-\xi)\!\left(\!{\rm Re}{\,}\dfrac{h_\sigma(tz)}{tz}-\dfrac{1-\sigma(1+t)}{(1-\sigma)(1+t)^2}\!\right)\!+\!
\xi\!\left(\!{\rm Re}{ \,} h'_\sigma(tz)-\dfrac{1-\sigma-(1+\sigma)t}{(1-\sigma)(1+t)^3}\!\right).
\end{align*}
The value of $\beta$ is sharp.
\end{theorem}
\begin{proof}
The case $\gamma=0$ was considered by the second author in \cite[Theorem 2.1]{rag
M}, so we consider here only the case $\gamma>0$. 
Consider
\begin{align*}
H(z):=(1-\alpha+2\gamma)\dfrac{f(z)}{z}+(\alpha-2\gamma)f^\prime (z) + \gamma z f^{\prime\prime}(z).
\end{align*}
Using \eqref{eq-mu+nu} in above equality gives
\begin{align}\label{H(z)-eq}
H(z)
=\mu \nu z^{1-1/\mu }\left(
z^{1/\mu -1/\nu +1}\left( z^{1/\nu -1}f(z)\right)'
\right)'.
\end{align}
Let $f(z)=z+\displaystyle\sum_{n=2}^{\infty} a_n z^n$. Now
using \eqref{phi-series} and \eqref{psi-series},
\eqref{H(z)-eq} is equivalent to
\begin{align}\label{f'(z)}
H(z)=f^\prime (z) \ast \phi_{\mu, \,\nu}(z)
\Rightarrow f^\prime (z)=H(z) \ast \psi_{\mu, \,\nu}(z).
\end{align}
Consider $G(z)\!:=\!(H(z)-\beta)/(1-\beta)$. Therefore
${\rm{Re}}(e^{i\phi} G(z))>0$.

Using the duality technique given in \cite{Rus} it is easy to see that
$G(z)=(1+xz)/(1+yz)$, where $|x|\!=\!|y|\!=\!1$. Hence
\begin{align} \label{H(z)-beta}
H(z)=\left((1-\beta)\left(\dfrac{1+xz}{1+yz}\right)+\beta\right).
\end{align}
From \eqref{f'(z)} and \eqref{H(z)-beta} gives
\begin{align*}
f^\prime (z)=\left((1-\beta)\left(\dfrac{1+xz}{1+yz}\right)+\beta\right)\ast \psi_{\mu, \,\nu}(z).
\end{align*}
Integrating the above expression gives
\begin{align*}
\dfrac{f(z)}{z}=\dfrac{1}{z}\int_0^z\!\left((1-\beta)\left(\dfrac{1+xw}{1+yw}\right)+\beta\right)dw \ast \psi_{\mu, \,\nu}(z).
\end{align*}
For $F\!\in\! M(\sigma, \xi)\Longleftrightarrow (\xi zF^\prime
+ (1-\xi)F) \in \mathcal{S}^\ast(\sigma)$. So it is sufficient
to check the condition of starlikeness. By the well known
result from convolution theory \cite[Pg 94]{Rus},
\begin{align*}
(\xi zF^\prime + (1-\xi)F) \in \mathcal{S}^\ast(\sigma)\Longleftrightarrow \dfrac{1}{z}\left(\dfrac{}{}((1-\xi)F + \xi zF^\prime)\ast h_\sigma(z) \right)\neq 0
\end{align*}
where
\begin{align*}
h_\sigma(z)=\frac{z\left(1+\dfrac{\epsilon+2\sigma-1}{2(1-\sigma)}z\right)}{(1-z)^2}, \quad |\epsilon|=1,\quad|z|<1.
\end{align*}
Now $F\in M(\sigma, \xi)$ if and only if
\begin{align*}
0
&\neq\dfrac{1}{z}\left((1-\xi)\left(\int_0^1\lambda(t)\dfrac{f(tz)}{t}{\,}dt\ast\dfrac{h_\sigma(z)}{z}\right)+\xi\left(\int_0^1\lambda(t)\dfrac{f(tz)}{t}{\,}dt\ast zh_\sigma'(z)\right)\right)\\
&=(1-\xi)\left(\int_0^1\dfrac{\lambda(t)}{1-tz}{\,}dt\ast\dfrac{f(z)}{z}\ast\dfrac{h_\sigma(z)}{z}\right)+\xi\left(\int_0^1\dfrac{\lambda(t)}{1-tz}{\,}dt\ast\dfrac{f(z)}{z}\ast h_\sigma'(z)\right)\\
&=(1-\xi)\left(\int_0^1\dfrac{\lambda(t)}{1-tz}{\,}dt \ast \left(\dfrac{1}{z}\int_0^z(1-\beta)\left(\dfrac{1+x\omega}{1+y\omega}\right){\,}d\omega+\beta\right)\ast\psi(z)\ast\dfrac{h_\sigma(z)}{z}\right)\\
& \quad\quad +\xi\left(\int_0^1 \dfrac{\lambda(t)}{1-tz}{\,}dt\ast \left(\dfrac{1}{z}\int_0^z(1-\beta)\left(\dfrac{1+x\omega}{1+y\omega}\right){\,}d\omega+\beta\right)\ast\psi(z)\ast h_\sigma'(z)\right)\\
&=(1-\xi)\left(\int_0^1\lambda(t)\dfrac{h_\sigma(tz)}{tz}{\,}dt \ast (1-\beta)\left(\dfrac{1}{z}\int_0^z\dfrac{1+x\omega}{1+y\omega}{\,}d\omega+\dfrac{\beta}{1-\beta}\right)\ast\psi(z)\right)\\
& \quad\quad +\xi\left(\int_0^1 \lambda(t)h_\sigma^\prime(tz){\,}dt\ast (1-\beta)\left(\dfrac{1}{z}\int_0^z\dfrac{1+x\omega}{1+y\omega}{\,}d\omega+\dfrac{\beta}{1-\beta}\right)\ast\psi(z)\right)\\
&=(1-\beta)\left[(1-\xi)\int_0^1\lambda(t)\left(\dfrac{1}{z}\int_0^z\dfrac{h_\sigma(t\omega)}{t\omega}{\,}d\omega+\dfrac{\beta}{1-\beta}\right){\,}dt\right. \\
&\quad\quad\quad\quad\quad\quad\left.+\xi\int_0^1\lambda(t)\left(\dfrac{1}{z}\int_0^z h_\sigma^\prime(t\omega)d\omega+\dfrac{\beta}{1-\beta}\right){\,}dt\right]
\ast\psi(z)\ast\dfrac{1+xz}{1+yz},
\end{align*}
which clearly holds if and only if \cite[Pg 23]{Rus}
\begin{align*}
{\rm Re}{\,}\left((1-\beta)\left[(1-\xi)\int_0^1\lambda(t)\left(\dfrac{1}{z}\int_0^z\dfrac{h_\sigma(t\omega)}{t\omega}d\omega+\dfrac{\beta}{1-\beta}\right)dt\right.\right. \\
\left.\left.+\xi\int_0^1\lambda(t)\left(\dfrac{1}{z}\int_0^z h_\sigma^\prime(t\omega)d\omega+\dfrac{\beta}{1-\beta}\right)dt\right]
\ast\psi(z)\right)>\dfrac{1}{2}.
\end{align*}
Using \eqref{sarika-W-class-S*-result}
and\eqref{sarika-W-class-C-result} from Theorem
\ref{sarika-W-class-S*} and Theorem \ref{sarika-W-class-C}, the
above inequality is equivalent to
\begin{align*}
\int_{0} ^{1} t^{{1/{\mu}} -1} \Pi_{\mu, \,\nu}(t)\left[(1-\xi) \left( {\rm Re}{\,} \dfrac{h_\sigma(tz)}{tz}-\dfrac{1-\sigma(1+t)}{(1-\sigma)(1+t)^2}\right)\right.\\
\left.+\xi \left({\rm Re}{\,} h'_\sigma(tz)-\dfrac{1-\sigma-(1+\sigma)t}{(1-\sigma)(1+t)^3}\right)\right]dt\geq 0,
\end{align*}
which directly implies that $F\in M(\sigma,
\xi)\Longleftrightarrow\mathcal{M}_{\Pi_{\mu,\,\nu}}(h_\sigma)\geq0$.

Now to prove the sharpness, let $f(z)\in \mathcal{W}_\beta(\alpha, \,\gamma)$ be the solution
of the differential equation
\begin{align}\label{sharpness-eq-pascu}
(1-\alpha+2\gamma)\dfrac{f(z)}{z}+(\alpha-2\gamma)f^\prime (z) + \gamma z f^{\prime\prime}(z)=\beta+(1-\beta)\dfrac{1+z}{1-z}
\end{align}
where $\beta$ satisfies \eqref{beta-cond-pascu}. Using series
expansion of \eqref{int-g(t)} and \eqref{int-q(t)} in
\eqref{beta-cond-pascu} gives {\small{
\begin{align}\label{sharpness-eq-pascu-1}
\dfrac{\beta}{1-\beta}=\!-\!1\!-\!\dfrac{2}{(1-\sigma)}\!\left[\!(1-\xi)\!\sum_{n=1}^\infty\dfrac{(n+1-\sigma)(-1)^n\tau_n}{(1+\mu n)(1+\nu n)}
\!+\!\xi\!\sum_{n=1}^\infty\dfrac{(n+1)(n+1-\sigma)(-1)^n\tau_n}{(1+\mu n)(1+\nu n)}\!\right]
\end{align}}}
where $\displaystyle\tau_n=\int_0^1\lambda(t){\,}t^n dt$.
\\
If
\begin{align}\label{conv-M-to-S*}
F:=V_\lambda(f(z))\!\in\! M(\sigma,\xi)\Longleftrightarrow K(z):=(\xi zF^\prime\!+\!(1-\xi)F)\!\in\! \mathcal{S}^\ast(\sigma).
\end{align}
Using \eqref{eq-mu+nu} and \eqref{sharpness-eq-pascu} gives
\begin{align*}
f(z)=z+\sum_{n=1}^{\infty }\dfrac{2(1-\beta )}{(n\mu +1)(n\nu+1)}z^{n+1},
\end{align*}
which further gives
\begin{align*}
F:=V_\lambda(f(z))=\int_0^1\lambda(t)\dfrac{f(tz)}{t}dt
=z+\sum_{n=1}^\infty\dfrac{2(1-\beta)\tau_n}{(n\mu+1)(n\nu+1)}z^{n+1}.
\end{align*}
The above equation implies that \small{{
\begin{align}\label{eq-k-sharpness}
K(z):=(1-\xi)\left(\!z+\sum_{n=1}^\infty\!\dfrac{2(1-\beta)\tau_n }{(n\mu+1)(n\nu+1)}z^{n+1}\!\right)\!+\!
\xi\left(\!z+\sum_{n=1}^\infty\!\dfrac{2(1-\beta)(n+1)\tau_n }{(n\mu+1)(n\nu+1)}z^{n+1}\!\right).
\end{align}}}
Using \eqref{sharpness-eq-pascu-1} and \eqref{eq-k-sharpness}
gives
\begin{align*}
zK^\prime(z)|_{z=-1}&=-1-(1-\xi)\sum_{n=1}^\infty\dfrac{2(1-\beta)(n+1)\tau_n (-1)^n}{(n\mu+1)(n\nu+1)}\\
&\quad\quad-\xi\sum_{n=1}^\infty\dfrac{2(1-\beta)(n+1)^2\tau_n (-1)^n}{(n\mu+1)(n\nu+1)}\\
&=\sigma K(-1).
\end{align*}
Therefore $(zK^\prime(z))/(K(z))=\sigma$ at $z=-1$, which
clearly indicates that the result is sharp.
\end{proof}
Particular values of Theorem $\ref{Thm-main-pascu}$ reduce to several known results.
\begin{remark}
\begin{enumerate}
\item For $\xi=0$, {\rm Theorem \ref{Thm-main-pascu}}
    reduces to {\rm \cite[Theorem 3.1]{sarika-S*}} {\rm
    (see also \cite[Theorem $2.1$]{Suzeini-starlike})} and
    for $\xi=1$, {\rm Theorem \ref{Thm-main-pascu}} gives
    \cite[Theorem 3.1]{sarika-C} {\rm (see also
    \cite[Theorem 2.2]{Suzeini-convex})}.
\item When $\sigma=0$, {\rm Theorem
    \ref{Thm-main-pascu}} reduces to {\rm \cite[Theorem
    2.1]{Swami M}}.
\end{enumerate}
\end{remark}
The necessary and sufficient conditions so that the integral
operator given in \eqref{eq-lambda-operator} carries the
function from $\mathcal{W}_{\beta}(\alpha,\gamma)$ into the
class $M(\sigma,\xi)$ is obtained in Theorem
\ref{Thm-main-pascu}, is not an easy on to use for the applications. Hence for the application purpose an
easier sufficient condition is presented in the following theorem.
\begin{theorem}\label{Thm-inc-pascu}
Let $\sigma\in [0,1/2]$ and assume that $\Lambda_\nu $ and
$\Pi_{\mu, \,\nu}$ are integrable on $[0,1]$ and positive on
$(0,1)$. If $\beta<1$ satisfy \eqref{beta-cond-pascu} and
\begin{align}\label{inc-cond-pascu}
\dfrac{\xi t^{{1 / {\xi}} -{1 / {\mu}} +1}\hspace{.2cm}d {\left(t^{{1/{\mu}}-{1/{\xi}}}{\,}\Pi_{\mu, \,\nu}(t)\right)}}{\left(\log(1/t)\right)^{1+2\sigma}}
\end{align}
is increasing on (0,1), for $\mu \geq 1$, $0 \leq \xi \leq 1$.
Then $V_{\lambda}(\mathcal{W}_{\beta}(\alpha,\gamma))\in
M(\sigma,\xi)$.
\end{theorem}
\begin{proof}
Consider
\begin{align*}
\mathcal{M}_{\Pi_{\mu, \,\nu}}(h_\sigma)=\int_0^1 t^{{1/{\mu}} -1} \Pi_{\mu, \nu}(t)
&\left[(1-\xi)\left({\rm{Re}}{\,}\dfrac{h_\sigma(tz)}{tz} - \dfrac {1-\sigma(1+t)}{(1-\sigma)(1+t)^2} \right)\right.\\
&\quad+\left.\xi\left({\rm{Re}}{\,}h_\sigma^{\prime} (tz) - \dfrac{(1-\sigma-(1-\sigma)t)}{(1-\sigma)(1+t)^3}\right)\right]dt
\end{align*}

\begin{align*}
\quad\quad\quad\quad=\int_{0}^{1}t^{{1/{\mu}}-1}\Pi_{\mu, \,\nu}(t)\left[(1-\xi)\left({\rm{Re}}{\,}\dfrac{h_\sigma(tz)}{tz} - \dfrac {1-\sigma(1+t)}{(1-\sigma)(1+t)^2} \right)\right.\\
\left.+\xi\dfrac{d}{dt}\left({\rm{Re}}\dfrac{h_\sigma(tz)}{tz} - \dfrac {t(1-\sigma(1+t))}{(1-\sigma)(1+t)^2} \right)\right]dt.
\end{align*}
A simple computation gives {\small{
\begin{align*}
\mathcal{M}_{\Pi_{\mu, \,\nu}}(h_\sigma)=\int_0^1t^{{1/\mu}-1}\left(\left(1-\dfrac{\xi}{\mu}\right)\Pi_{\mu, \,\nu}(t)-\xi t\Pi_{\mu, \,\nu}^\prime(t)\right)
\left({\rm{Re}}\dfrac{h_\sigma(tz)}{tz}-\dfrac{1-\sigma(1+t)}{(1-\sigma)(1+t)^2}\right)dt.
\end{align*}
}} The function $t^{1/\mu-1}$ decreases on $(0,1)$, when
$\mu\geq1$. Therefore the condition \eqref{inc-cond-pascu}
along with \cite[Theorem 1.3]{pons-star} gives
$\mathcal{M}_{\Pi_{\mu, \,\nu}}(h_\sigma)\geq 0$. So the
desired result follows from Theorem \ref{Thm-main-pascu}.
\end{proof}
To ensure the sufficiency of Theorem \ref{Thm-inc-pascu} for
the integral transform to be in $M(\sigma,\xi)$ by an easier
method, the following results are obtained.

Since the case $\gamma=0$ was considered in \cite{rag M}, we only consider the case $\gamma>0$.
In \cite{rag M}, for the case $\gamma=0$ the condition $\lambda(1)=0$ was assumed.
To prove that the Theorem
    \ref{Thm-main-pascu} holds true for the case
    $\gamma>0$, we need to show that the condition
\begin{align*}
p(t):=\dfrac{\left(1-\dfrac{\xi}{\mu}\right)\Pi_{\mu, \,\nu}(t)+\xi t^{1/\nu-1/\mu}\Lambda_\nu(t)}{\left(\log(1/t)\right)^{1+2\sigma}}
\end{align*}
is decreasing on (0,1), where $\Lambda_\nu(t)$ and
$\Pi_{\mu, \,\nu}(t)$ are defined in
$\eqref{eqn-lambda-nu}$ and $\eqref{eqn-Pi-nu-mu}$. For this it is enough to prove that $p ^\prime(t)\leq 0$.\\
Now $p(t)=k(t)/(\log(1/t))^{1+2\sigma}$ where
$k(t)=\left(1-\dfrac{\xi}{\mu}\right)\Pi_{\mu,
\,\nu}(t)+\xi t^{1/\nu - 1/\mu}\Lambda_\nu(t)$. The
condition $p'(t)\leq0$ is equivalent of obtaining
\begin{align*}
\dfrac{p'(t)}{p(t)}=\dfrac{k'(t)}{k(t)}+\dfrac{(1+2\sigma)}{t(\log(1/t))}\leq 0,
\end{align*}
since $p(t)\geq0$ for $t\in(0,1)$. In order to show that
$p'(t)\leq0$, is similar to obtain
\begin{align*}
r(t):&=k(t)+\dfrac{t\log(1/t){\,}k'(t)}{(1+2\sigma)}\\
&=\left(1-\dfrac{\xi}{\mu}\right)\Pi_{\mu, \,\nu}(t)+\xi t^{1/\nu - 1/\mu}\Lambda_\nu(t)\\
&\quad\quad-\dfrac{\log(1/t)\left(\left(1-\dfrac{\xi}{\nu}\right)t^{1/\nu-1/\mu}\Lambda_\nu(t)+\xi t^{1-1/\mu}\lambda(t)\right)}{(1+2\sigma)}\leq0,\quad t\in(0,1).
\end{align*}
As $r(1)=0$, in order to prove that $p'(t)\leq0$, it is
enough to show that $r(t)$ is an increasing function on
$(0,1)$. We compute $r'(t)$ explicitly and after an easy
computation $r'(t)\geq0$ is equivalent to the inequality
{\small{
\begin{align}\label{eq-pascu-sigma-suff-s(t)}
\quad\quad s(t)
&:=\left(1-\dfrac{\xi}{\nu}\right)
\left(\left(\dfrac{1}{\mu}-\dfrac{1}{\nu}\right)\log(1/t)-2\sigma\right)t^{1/\nu-1-1/\mu}\Lambda_\nu(t)\nonumber\\
&\!+\!\left(\!\left(\!1\!-\!\xi+\dfrac{\xi}{\mu}-\dfrac{\xi}{\nu}\right)\log(1/t)-2\sigma\xi\right)t^{-1/\mu}\lambda(t)
-\xi\log(1/t)t^{1-1/\mu}\lambda'(t)\geq0.
\end{align}
}} At $t=1$, $s(t)\leq0$ because $\xi$ and $\sigma$ are
positive terms. So we assume that $\lambda(t)=0$.
Hence $s(t)=0$ at $t=1$. The function $s(t)$ is positive if $s(t)$ is a deceasing function of $t\in(0,1)$ i.e., $s'(t)\leq0$.\\
Using \eqref{eq-pascu-sigma-suff-s(t)}, $s'(t)$ is equivalent to {\small{
\begin{align*}
\quad\quad\quad s'(t)=L(t)t^{1/\nu-2-1/\mu}\Lambda_\nu(t)+M(t)t^{-1-1/\mu}\lambda(t)+N(t)t^{-1/\mu}\lambda'(t)+P(t)t^{1-1/\mu}\lambda''(t)\leq0
\end{align*}
}} where
\begin{align*}
L(t):=& \left(1-\dfrac{\xi}{\nu}\right)\left[\left(1+\dfrac{1}{\mu}-\dfrac{1}{\nu}\right)\left(2\sigma-\left(\dfrac{1}{\mu}-\dfrac{1}{\nu}\right)\log(1/t)\right)
-\left(\dfrac{1}{\mu}-\dfrac{1}{\nu}\right)\right]\\
M(t):= & -\!\!\left(\!\left[(1-\xi)\dfrac{1}{\mu}+\left(\dfrac{1}{\mu}-\dfrac{1}{\nu}\right)\left(1+\dfrac{\xi}{\mu}-\dfrac{\xi}{\nu}\right)
\right]\log(1/t)
+(1-2\sigma)\!\left(\!1+\dfrac{\xi}{\mu}-\dfrac{\xi}{\nu}\right)-\xi\right)\\
N(t):=& \left(\left(1-2\xi+\dfrac{2\xi}{\mu}-\dfrac{\xi}{\nu}\right)\log(1/t)+\xi(1-2\sigma)\right)
\end{align*}
and $P(t):=-\xi\log(1/t)$.

From \eqref{eq-mu+nu}, $\mu\geq1$ implies
$\nu\geq\mu\geq1$. For $\xi\in[0,1]$, the term
$(1-\xi/\nu)$ and $t^{1/\nu-1-1/\mu}\Lambda_\nu(t)$ are
non-negative for $t\in(0,1)$. The function $s'(t)\leq0$, if
$L(t)$, $M(t)$, $N(t)$ and $P(t)$ are negative. For
\begin{align*}
\sigma\leq\dfrac{1}{2}\left(\dfrac{1/\mu-1/\nu}{1+1/\mu-1/\nu}\right),\quad\quad \mu\geq1,
\end{align*}
implies $L(t)\leq0$ and $M(t)\leq 0$.

Now to show that $s'(t)\leq0$ for $t\in(0,1)$, it is enough
to prove that {\small{
\begin{align}\label{sigma-cond-geq0}
\xi t\log(1/t)\lambda''(t)
&-\left(\left(1-2\xi+\dfrac{2\xi}{\mu}-\dfrac{\xi}{\nu}\right)\log(1/t)+\xi(1-2\sigma)\right)\lambda'(t)\geq0\\
&\Longleftrightarrow\quad\dfrac{t\lambda''(t)}{\lambda'(t)}\geq\left(\dfrac{1}{\xi}-2+\dfrac{2}{\mu}-\dfrac{1}{\nu}\right)
+\dfrac{(1-2\sigma)}{\log(1/t)}\nonumber,
\end{align}
}} for $\mu\geq1$ and
$\sigma\leq\dfrac{1}{2}\left(\dfrac{1/\mu-1/\nu}{1+1/\mu-1/\nu}\right)$.

Now we are in a position to state the general result.
\begin{theorem}\label{thm-suff-pascu}
Let $\beta<1$ is defined by \eqref{beta-cond-pascu}. Then
$V_\lambda(\mathcal{W}_\beta(\alpha,\gamma))\!\in\!
M(\sigma,\xi)$, if
\begin{align}\label{suff-cond-pascu}
\dfrac{t\lambda''(t)}{\lambda'(t)}\geq\left(\dfrac{1}{\xi}-2+\dfrac{2}{\mu}-\dfrac{1}{\nu}\right)
+\dfrac{(1-2\sigma)}{\log(1/t)},
\end{align}
for $\mu\geq1$, $\gamma>0$ and
$\sigma\leq\dfrac{1}{2}\left(\dfrac{1/\mu-1/\nu}{1+1/\mu-1/\nu}\right)$.
\end{theorem}
\section{applications}\label{sec-application}
In this section, the number of applications for well-known
integral operators are considered and the conditions are
obtained.
\begin{remark}
Note that $\lambda(1)=0$ is assumed in Theorem $\ref{Thm-inc-pascu}$. For the case
$ \lambda(t) = (c+1)t^{c-1}$, it is not possible to get $\lambda(1)=0$ and hence
the result corresponding to Bernardi integral transform cannot be obtained using Theorem $\ref{Thm-inc-pascu}$.
But the admissibility $V_\lambda(\mathcal{W}_\beta (\alpha,{\,} \gamma))\!\in\! M(0,\xi)$
for the Bernardi integral operator was possible in \cite[Theorem 12]{Swami M}, as this condition was not required in
the admissibility to the class of $M(0,\xi)$ in \cite{Swami M}.
\end{remark}

Consider
\begin{align}\label{generalize-lambda}
\lambda(t)=D t^{B-1}(1-t)^{C-A-B}\omega(1-t),
\end{align}
where the function
$\omega(1-t)=1+\displaystyle\sum_{n=1}^\infty x_n(1-t)^n$, for
$t\in(0,1)$ and $x_n\geq0$. $D$ is chosen such that it
satisfies normalization condition
$\displaystyle\int_0^1\lambda(t)dt=1$. An easy computation on
$\lambda(t)$ defined by \eqref{generalize-lambda}, obtain
$\lambda'(t)$ and $\lambda''(t)$. Substituting in
\eqref{sigma-cond-geq0} will give {\small{
\begin{align*}
Dt^{B-2}(1-t)^{C-A-B-2}
[X(t)\omega(1-t)+t(1-t)Y(t)\omega'(1-t)+t^2(1-t)^2\omega''(1-t)]\geq0,
\end{align*}}}
for $t\in(0,1)$, where {\small{
\begin{align*}
X(t):=\log(1/t)((B-1)(B-2)+2(1-B)(C-A-2)t+(C-A-B)(C-A+B-3)t^2)\\
+((C-A-B)t+(1-B)(1-t))\left((1-2\sigma)+\left(\dfrac{1}{\xi}-2+\dfrac{2}{\mu}-\dfrac{1}{\nu}\right)\log(1/t)\right)
\end{align*}}}
{\small{
\begin{align*}
Y(t):=2\log(1/t)((C-A-B)t+(1-B)(1-t))
\!+\!\left(\!(1-2\sigma)\!+\left(\!\dfrac{1}{\xi}-2+\dfrac{2}{\mu}-\dfrac{1}{\nu}\!\right)\log(1/t)\!\right).
\end{align*}}}
The inequality \eqref{suff-cond-pascu} will hold true, if
$X(t)$ and $Y(t)$ are non-negative terms for some values of
$A$, $B$ and $C$. When $B\leq1$, $C\geq A+3$ and
$\left(\dfrac{1}{\xi}+\dfrac{2}{\mu}-\dfrac{1}{\nu}\right){\,}\geq{\,}2$,
then Theorem \ref{thm-suff-pascu} is true.

The above observation results in the following theorem.
\begin{theorem}\label{Thm-generalize-pascu}
For $\gamma>0$, $\mu\geq1$, $\xi\in[0,1]$ and $\beta<1$ is
defined by \eqref{beta-cond-pascu}, where $\lambda(t)=D
t^{B-1}(1-t)^{C-A-B}\omega(1-t)$. Then
$V_\lambda(\mathcal{W}_\beta(\alpha,\gamma))\in M(\sigma,\xi)$,
if $B\leq1$, $C\geq A+3$ and
$\left((1/\xi)+(2/\mu)-(1/\nu)\right)\geq2$, for
\begin{align*}
0{\,}\leq{\,}\sigma\leq{\,}\dfrac{1}{2}\left(\dfrac{1/\mu-1/\nu}{1+1/\mu-1/\nu}\right).
\end{align*}
\end{theorem}
\begin{theorem}\label{Thm-hohlov-pascu}
Let $a, b, c>0$, $\gamma>0$ and  $\beta<1$ satisfy
\begin{align}\label{beta-cond-hohlov}
\dfrac{\beta}{(1-\beta)}=-R\int_0^1
& t^{b-1}(1-t)^{c-a-b}\nonumber\\
&{\,}_{2}F_1\left(
\begin{array}{cll}&\displaystyle c-a,\quad 1-a
\\
&\displaystyle c-a-b+1
\end{array};1-t\right)
[(1-\xi)g(t)+\xi(2q(t)-1)]dt,
\end{align}
where $R=\Gamma(c)/(\Gamma(a)\Gamma(b)\Gamma(c-a-b+1))$. Then
$V_\lambda(\mathcal{W}_{\beta}(\alpha,\gamma))\in
M(\sigma,\xi)$ if $c\geq a+3$, $b\leq1$,
$\left(\dfrac{1}{\xi}+\dfrac{2}{\mu}-\dfrac{1}{\nu}\right)\geq2{\,}$
and
${\,}0\leq\sigma\leq\dfrac{1}{2}\left(\dfrac{1/\mu-1/\nu}{1+1/\mu-1/\nu}\right)$.
The value of $\beta$ is sharp.
\end{theorem}
\begin{proof}
Choosing
\begin{align*}
\omega(1-t)=
R{\,}{\,} {\,}_{2}F_1\left(
\begin{array}{cll}&\displaystyle c-a,\quad 1-a
\\
&\displaystyle c-a-b+1
\end{array};1-t\right),
\end{align*}
and substitute $a$, $b$, $c$ instead of $A$, $B$, $C$ respectively in Theorem \ref{Thm-generalize-pascu} will give the required result.\\
In order to obtain sharpness take the extremal function $f(z)$
of the class $\mathcal{W}_\beta(\alpha,\gamma)$ as
\begin{align*}
f(z)=z+2\sum_{n=2}^{\infty}\dfrac{(1-\beta)}{(1+(n-1)\mu)(1+(n-1)\nu)}z^{n}.
\end{align*}
Consider
\begin{align*}
\Omega_n:=\int_0^1 t^{n+b-2}(1-t)^{c-a-b}
{\,}_{2}F_1\left(
\begin{array}{cll}&\displaystyle c-a,\quad 1-a
\\
&\displaystyle c-a-b+1
\end{array};1-t\right) dt.
\end{align*}
Using \eqref{eq-lambda-operator} and
\begin{align*}
\lambda(t):=R{\,}t^{b-1}(1-t)^{c-a-b}{\,}{\,}_{2}F_1\left(
\begin{array}{cll}&\displaystyle c-a,\quad 1-a
\\
&\displaystyle c-a-b+1
\end{array};1-t\right),
\end{align*}
it follows that
\begin{align*}
F(z):=z+2(1-\beta)R{\,}\sum_{n=2}^\infty\dfrac{\Omega_n}{(1+(n-1)\mu)(1+(n-1)\nu)}z^n,
\end{align*}
which implies that
\begin{align*}
zF'(z)=z+2(1-\beta)R{\,}\sum_{n=2}^\infty\dfrac{n\Omega_n}{(1+(n-1)\mu)(1+(n-1)\nu)}z^n.
\end{align*}
Using \eqref{conv-M-to-S*} gives
\begin{align}\label{sharpness-eq-komatu-1}
K(z)=z+2(1-\beta)R{\,}{\,}\sum_{n=2}^\infty\dfrac{((n-1)\xi+1)\Omega_n}{(1+(n-1)\mu)(1+(n-1)\nu)}z^n.
\end{align}
For the case $\gamma>0$, the series representation of $g(t)$
defined in \eqref{int-g(t)} is given by
\begin{align}\label{series-g(t)}
g(t)=2\sum_{n=0}^\infty\dfrac{(n+1-\sigma)(-t)^n}{(1-\sigma)(1+n\mu)(1+n\nu)}-1,
\end{align}
and $q(t)$ defined in \eqref{int-q(t)} is given by
\begin{align}\label{series-q(t)}
q(t)=\sum_{n=0}^\infty\dfrac{(n+1)(n+1-\sigma)(-t)^n}{(1-\sigma)(1+n\mu)(1+n\nu)}.
\end{align}
From \eqref{series-g(t)} and \eqref{series-q(t)}\\
$(1-\xi)g(t)+\xi(2q(t)-1)$ {\small{
\begin{align}\label{series-g(t)+q(t)}
=1+2\sum_{n=2}^\infty\dfrac{(1+\xi(n-1))(n-\sigma)(-t)^{n-1}}{(1-\sigma)(1+(n-1)\mu)(1+(n-1)\nu)}\hspace{5cm}
\end{align}
\begin{align}\label{conv-g(t)+q(t)-gaussian}
\quad\quad=2{\,}{\,}_5F_4\left(1,\dfrac{1}{\mu},\dfrac{1}{\nu},(2-\sigma),\left(1+\dfrac{1}{\xi}\right){\,};
\left(1+\dfrac{1}{\mu}\right),\left(1+\dfrac{1}{\nu}\right),(1-\sigma),\dfrac{1}{\xi}{\,};{\,}-t\right)-1.
\end{align}}}
Using \eqref{beta-cond-hohlov} and \eqref{series-g(t)+q(t)} gives
\begin{align}\label{sharpness-eq-komatu-2}
\dfrac{1}{2(1-\beta)}=R{\,}{\,}\sum_{n=2}^\infty\dfrac{(1+\xi(n-1))(n-\sigma)(-1)^{n+1}}{(1-\sigma)(1+(n-1)\mu)(1+(n-1)\nu)}\Omega_n.
\end{align}
From \eqref{sharpness-eq-komatu-1} and
\eqref{sharpness-eq-komatu-2}, as $z\rightarrow-1$ gives
\begin{align*}
\left.\dfrac{zK'(z)}{K(z)}\right|_{z=-1}=\sigma,
\end{align*}
which means that the result is sharp.
\end{proof}
\begin{remark}
\begin{enumerate}\item[]
\item When $\sigma=0$, Theorem $\ref{Thm-hohlov-pascu}$
    gives different conditions and a precise bound on $b$
    than in \cite [Theorem $3.4$]{Swami M}.
\item A particular instance, when $\xi=0$, Theorem
    $\ref{Thm-hohlov-pascu}$ gives a result with the
    smaller range for $\sigma$ than the result given in
    \cite[Theorem $5.1$]{sarika-S*} (see also
    {\rm\cite[Theorem 3.3]{Suzeini-starlike}}).
\end{enumerate}
\end{remark}
The result for the case, $\gamma=0$, $\sigma=0$ and $\xi=0$ is
obtained in \cite[Theorem 1]{saigo}. In Theorem
$\ref{Thm-hohlov-pascu}$, the conditions are obtained when
$\gamma>0$, hence the results cannot be compared.
\begin{corollary}
Let $a, b, c>0$, $\gamma>0$ and  $\beta_0<\beta<1$, where
\begin{align*}
\beta_0=1-\dfrac{1}{2\left(1-{\,}
{\,}_6F_5
\left(
\begin{array}{cll}&\displaystyle \quad\quad 1,b,\dfrac{1}{\mu},\dfrac{1}{\nu},(2-\sigma),\left(1+\dfrac{1}{\xi}\right)
\\
&\displaystyle c,\left(1+\dfrac{1}{\mu}\right),\left(1+\dfrac{1}{\nu}\right),(1-\sigma),\dfrac{1}{\xi}
\end{array}{\,};{\,}-1\right)
\right)}.
\end{align*}
Then $V_\lambda(\mathcal{W}_{\beta}(\alpha,\gamma))\in
M(\sigma,\xi)$ if $c\geq 4$, $b\leq1$,
$\left(\dfrac{1}{\xi}+\dfrac{2}{\mu}-\dfrac{1}{\nu}\right)
\geq2{\,}$ and
${\,}0\leq\sigma\leq\dfrac{1}{2}\left(\dfrac{1/\mu-1/\nu}{1+1/\mu-1/\nu}\right)$.
\end{corollary}
\begin{proof}
Consider $a=1$ in \eqref{beta-cond-hohlov}. Using
\eqref{conv-g(t)+q(t)-gaussian} gives
\begin{align*}
\dfrac{\beta}{1-\beta}=-\dfrac{\Gamma(c)}{\Gamma{(b)}\Gamma{(c-b)}}\int_0^1t^{b-1}(1-t)^{c-b-1}\hspace{6cm}\\
\quad\left( 2{\,}_5F_4
\left(
\begin{array}{cll}&\displaystyle \quad\quad 1,\dfrac{1}{\mu},\dfrac{1}{\nu},(2-\sigma),\left(1+\dfrac{1}{\xi}\right)
\\
&\displaystyle \left(1+\dfrac{1}{\mu}\right),\left(1+\dfrac{1}{\nu}\right),(1-\sigma),\dfrac{1}{\xi}
\end{array}{\,};{\,}-t\right) -1\right) dt,
\end{align*}
which is equivalent to
\begin{align*}
\dfrac{\beta-1/2}{1-\beta}=-\dfrac{\Gamma(c)}{\Gamma{(b)}\Gamma{(c-b)}}\int_0^1t^{b-1}(1-t)^{c-b-1}\hspace{6cm}\\
\quad{\,}_5F_4
\left(
\begin{array}{cll}&\displaystyle \quad\quad 1,\dfrac{1}{\mu},\dfrac{1}{\nu},(2-\sigma),\left(1+\dfrac{1}{\xi}\right)
\\
&\displaystyle \left(1+\dfrac{1}{\mu}\right),\left(1+\dfrac{1}{\nu}\right),(1-\sigma),\dfrac{1}{\xi}
\end{array}{\,};{\,}-t\right) dt
\end{align*}
\begin{align*}
=-\sum_{n=0}^\infty\dfrac{(1)_n{\,}(b_n){\,}\left(\dfrac{1}{\mu}\right)_n{\,}\left(\dfrac{1}{\nu}\right)_n{\,}(2-\sigma)_n{\,}\left(1+\dfrac{1}{\xi}\right)_n}
{(c_n){\,}\left(1+\dfrac{1}{\mu}\right)_n{\,}\left(1+\dfrac{1}{\nu}\right)_n{\,}(1-\sigma)_n{\,}\left(\dfrac{1}{\xi}\right)_n{\,}(1)_n}(-1)^n\\
=-{\,}_6F_5
\left(
\begin{array}{cll}&\displaystyle \quad\quad 1,b,\dfrac{1}{\mu},\dfrac{1}{\nu},(2-\sigma),\left(1+\dfrac{1}{\xi}\right)
\\
&\displaystyle c,\left(1+\dfrac{1}{\mu}\right),\left(1+\dfrac{1}{\nu}\right),(1-\sigma),\dfrac{1}{\xi}
\end{array}{\,};{\,}-1\right).
\end{align*}
By the given hypothesis and applying Theorem
\ref{Thm-main-pascu} will give the required result.
\end{proof}
\begin{theorem}\label{Thm-komatu-pascu}
Let $-1<c\leq 0$, $\delta\geq (3-c)$, $\gamma>0(\mu\geq1)$ and
$\beta<1$ satisfy
\begin{align*}
\dfrac{\beta}{(1-\beta)}=-\dfrac{(1+c)^\delta}{\Gamma(\delta)}\int_0^1
& t^{c}(\log(1/t))^{\delta-1}
[(1-\xi)g(t)+\xi(2q(t)-1)]dt.
\end{align*}
Then $V_\lambda(\mathcal{W}_{\beta}(\alpha,\gamma))\in
M(\sigma,\xi)$, for
\begin{align*}
\left(\dfrac{1}{\xi}+\dfrac{2}{\mu}-\dfrac{1}{\nu}\right){\,}\geq{\,}2\quad{\rm{and}}
\quad0\leq\sigma\leq\dfrac{1}{2}\left(\dfrac{1/\mu-1/\nu}{1+1/\mu-1/\nu}\right).
\end{align*}
\end{theorem}
\begin{proof}
Consider $B=c+1$, $C-A-B=\delta-1$ and
\begin{align*}
\omega(1-t)=\left(\dfrac{\log(1/t)}{(1-t)}\right)^{\delta-1}.
\end{align*}
Then the required conclusion follows from Theorem
\ref{Thm-hohlov-pascu}.
\end{proof}
\begin{remark}
\begin{enumerate}\item[]
\item For a particular value of $\sigma=0$ \cite[Theorem
    3.6]{Swami M} give weak result for $\delta$ in
    comparison to the result obtained from Theorem
    $\ref{Thm-komatu-pascu}$.
\item When $\xi=0$, \cite[Theorem 5.4]{sarika-S*} (see also
    {\rm\cite[Theorem 3.2]{Suzeini-starlike}}) give weak
    bounds for the parameters than the result obtained from
    Theorem $\ref{Thm-komatu-pascu}$.
\item When $\xi=1$, Theorem $\ref{Thm-komatu-pascu}$,
    improves the result obtained in \cite[Theorem
    5.9]{sarika-C}.
\end{enumerate}
\end{remark}
\begin{theorem}\label{Thm-pons-pascu}
Let $\mu\geq 1$, $\gamma>0$, $0<\xi\leq1$ and $\beta<1$ satisfy
\eqref{beta-cond-pascu}, where
\begin{align}\label{ponn-oper}
\lambda(t)= \begin{cases} (a+1)(b+1)\dfrac{t^a(1-t^{b-a})}{b-a},
 &b\neq a,\\
(a+1)^2t^a\log(1/t),  &  b=a.
\end{cases}
\end{align}
Then $V_\lambda(\mathcal{W}_{\beta}(\alpha,\gamma))\in
M(\sigma,\xi)$, provided $-1<a\leq0$ for $a= b$ or $a\neq b$
and
\begin{align*}
0\leq\sigma\leq\dfrac{1}{2}\left(\dfrac{1/\mu-1/\nu}{1+1/\mu-1/\nu}\right).
\end{align*}
\end{theorem}
\begin{proof}
Using $\lambda(t)$ given in \eqref{ponn-oper} gives
\begin{align*}
\dfrac{t\lambda''(t)}{\lambda'(t)}= \begin{cases} \dfrac{\left({\,}a(a-1)-b(b-1)t^{b-a}{\,}\right)}{\left({\,}a-bt^{b-a}{\,}\right)},
 &b\neq a,\\
\dfrac{(1-2a{\,}+{\,}a(a-1)\log(1/t))}{(-1{\,}+{\,}a\log(1/t))},  &  b=a.
\end{cases}
\end{align*}

Case(i): Consider $a=b>-1$ and $\gamma>0$. Substituting the
values of $t\lambda''(t)/\lambda'(t)$ in
\eqref{suff-cond-pascu} and on further simplification gives
\begin{align}\label{eq-pons1}
U(t)(\log(1/t))^2{\,}+{\,}V(t)(\log(1/t)){\,}+{\,}(1-2\sigma)\geq0,
\end{align}
where

$U(t):=a\left(a+1-\dfrac{1}{\xi}-\dfrac{2}{\mu}+\dfrac{1}{\nu}\right)$,

$V(t):=\left(\dfrac{1}{\xi}+\dfrac{2}{\mu}-\dfrac{1}{\nu}-1-a(3-2\sigma)\right)$.

As $(1-2\sigma)>0$, then \eqref{eq-pons1} holds true if $U(t)$
and $V(t)$ are non-negative on the given conditions. Clearly
the inequality \eqref{eq-pons1} under the hypotheses is true.

Case(ii): Consider $-1<a<b$ and $\gamma>0$. Substituting the
value of $t\lambda''(t)/\lambda'(t)$ in \eqref{suff-cond-pascu}
is equivalent to $\phi_t(a)\geq\phi_t(b)$, $t\in(0,1)$, where
\begin{align*}
\phi_t(a):=a(a-1)t^a\log(1/t)-a\left(\left(\dfrac{1}{\xi}+\dfrac{2}{\mu}-\dfrac{1}{\nu}-2\right)\log(1/t)+(1-2\sigma)\right)t^a.
\end{align*}
We will claim that $\phi_t(a)$ is a decreasing function of $a$
for $a\in (-1,0]$. Differentiating $\phi_t(a)$ with respect to
$a$ gives
\begin{align*}
\phi_t'(a):=-t^a({\,}R{\,}+{\,}S{\,}\log(1/t){\,}+{\,}T{\,}(\log(1/t))^2)
\end{align*}
where

$R:=(1-2\sigma),$

$S:=\left(\dfrac{1}{\xi}+\dfrac{2}{\mu}-\dfrac{1}{\nu}-1+a(2\sigma-3)\right),$

$T:=a\left(a-\dfrac{1}{\xi}-\dfrac{2}{\mu}+\dfrac{1}{\nu}+1\right).$\\
The function $\phi_t'(a)\leq0$, if the terms $R$, $S$ and $T$
are non-negative, which is clearly true by the hypothesis when
$a\in(0,1].$ Hence the desired conclusion follows.
\end{proof}
\begin{remark}
\begin{enumerate}\item[]
\item In Theorem $\ref{Thm-pons-pascu}$, consider the case
    when $\xi=0$. We obtain the bound for $a$ as,
$-1{\,}<{\,}a{\,}\leq{\,}0;{\,\,} a=b {\,\,\rm or
\,\,}{\,\,}a{\,}<{\,}b$. This range for $a$ is smaller than
the range obtain in \cite[Theorem 5.3]{sarika-S*} (see also
    {\rm\cite[Theorem 3.2]{Suzeini-starlike}}). For the
    case when $\xi=1$, Theorem $\ref{Thm-pons-pascu}$
    provides weaker bounds for $a$ when compared to
    {\rm\cite[Theorem 5.6]{sarika-C}}.
\item Consider $\xi=1$ and $\sigma=0$. The result obtained
    by Theorem $\ref{Thm-pons-pascu}$ has same bound for
    the case $a=b\leq0$ and the result for $a<b$ coincides
    in \cite[Theorem 5.7]{Mahnaz C}.
\end{enumerate}
\end{remark}

\begin{theorem}\label{Thm-ali-singh-pascu}
Let $\mu\geq 1$, $\gamma>0$, $0<\xi\leq1$ and $\beta<1$ satisfy
\eqref{beta-cond-pascu}, where
\begin{align}\label{ali-singh-operator}
\lambda(t)=\dfrac{(1-k)(3-k)}{2}t^{-k}(1-t^2),\quad 0\leq k<1.
\end{align}
Then $V_\lambda(\mathcal{W}_{\beta}(\alpha,\gamma))\in
M(\sigma,\xi)$, if $k=(1-1/\xi-2/\mu+1/\nu)$ and $\sigma=1/2$.
\end{theorem}
\begin{proof}
Using $\lambda(t)$ given in \eqref{ali-singh-operator}, we have
\begin{align*}
\dfrac{t\lambda''(t)}{\lambda(t)}=-\dfrac{[k(1+k)-(2-k)(1-k)t^2]}{[k+(2-k)t^2]}.
\end{align*}
Therefore to obtain the result, it suffice to prove that
\begin{align*}
-\dfrac{[k(1+k)-(2-k)(1-k)t^2]}{[k+(2-k)t^2]}\geq\left(\dfrac{1}{\xi}-2+\dfrac{2}{\mu}-\dfrac{1}{\nu}\right)
+\dfrac{(1-2\sigma)}{\log(1/t)},
\end{align*}
under the given hypothesis. Using the fact that
$\log(1/t)\geq2(1-t)/(1+t)$, for $t\in(0,1)$ the result can be
easily obtained.

\end{proof}

\end{document}